\documentclass[envcountsame,envcountsect]{svmult}

\usepackage{amsfonts,amsmath,amssymb}
\usepackage{txfonts}
\usepackage{helvet}         
\usepackage{courier}        
\usepackage[bottom]{footmisc}

\def\deq{\triangleq}

\def\d{{\mathrm d}}
\def\E{{\mathbb E}}
\def\PP{{\mathbb P}}

\def\1{{\mathbf 1}}

\def\TV{{\text{\tiny TV}}}

\def\bs{{\backslash}}

\newcommand{\tvnrm}[1]{{\left\Vert #1 \right\Vert}_{\TV}}
\newcommand{\hide}[1]{}

\def\cB{{\cal B}}
\def\cC{{\cal C}}

\def\cF{{\cal F}}

\def\cL{{\cal L}}

\def\sE{{\mathsf E}}

\def\sX{{\mathsf X}}

\def\Reals{{\mathbb R}}
\def\R{\Reals}

\def\bdelta{{\boldsymbol{\delta}}}
\newcommand{\trn}{\ensuremath{{\scriptscriptstyle{\text T}}}}

\newcommand{\set}[1]{\left\{ #1 \right\}}

\usepackage{bbm}

\smartqed

\begin{document}

\title*{Concentration of Measure without Independence: a Unified Approach via the Martingale Method}
\titlerunning{Concentration without independence}

\author{Aryeh Kontorovich and Maxim Raginsky}

\institute{Aryeh Kontorovich \at Department of Computer Science, Ben-Gurion University, Beer-Sheva, Israel, \email{karyeh@cs.bgu.ac.il}
\and Maxim Raginsky \at Department of Electrical and Computer Engineering and the Coordinated Science Laboratory, University of Illinois, Urbana, IL, USA, \email{maxim@illinois.edu}
}

\maketitle

\abstract{The concentration of measure phenomenon
may be summarized as follows: a function of many weakly
dependent random variables that is not too sensitive to any
of its individual arguments will tend to take values very close
to its expectation. This phenomenon is most completely understood
when the arguments are mutually independent random variables, and there exist several powerful complementary methods for proving concentration inequalities, such as the martingale method, the entropy method, and the method of transportation inequalities. The setting of dependent arguments is much less well understood. This chapter focuses on the martingale method for deriving concentration inequalities without independence assumptions. In particular, we use the machinery of so-called Wasserstein matrices to show that the Azuma-Hoeffding concentration inequality for martingales with almost surely bounded differences,
  when applied in a sufficiently abstract setting, is powerful enough
  to recover and sharpen several known
  concentration results for nonproduct measures. Wasserstein matrices provide a natural formalism
  for capturing the interplay between the metric
  and the probabilistic structures, which is fundamental to the concentration phenomenon.}

\section{Introduction}

At its most abstract, the concentration of measure phenomenon
may be summarized as follows: a function of several weakly
dependent random variables that is not too sensitive to any
of the individual arguments will tend to take values very close
to its expectation. This phenomenon is most completely understood
in the case of independent arguments, and the recent book
\cite{Boucheron-book-2013} provides an excellent survey (see also \cite{Raginsky_Sason_FnT} for an exposition from the viewpoint of, and with applications to, information theory).

The case of dependent arguments has yet to mature into such a unified, overarching
theory. The earliest concentration results for non-product measures
were established for Haar measures on various groups, and
relied strongly on the highly symmetric nature of the Haar measure in question. These results
include L\'evy's classic isoperimetric inequality on the sphere
\cite{levy51} and
Maurey's
concentration inequality on the permutation group
\cite{maurey79}.
To the best of our knowledge, the first concentration result for a non-product, non-Haar measure
is due to Marton \cite{marton96},
where she proved a McDiarmid-type bound for contracting Markov chains.
A flurry of activity followed. Besides Marton's own follow-up work
\cite{marton98,marton03,marton04},
the transportation method she pioneered was extended by Samson
\cite{samson00}, and martingale techniques
\cite{MR1771956,chazottes07,kontram06},
as well as methods relying on the
Dobrushin interdependence matrix
\cite{kulske03,chatterjee05,wu06}, have been employed in obtaining
concentration results for non-product measures.
The underlying theme is that the independence assumption may be relaxed
to one of {\em weak} dependence, the latter being quantified by
various mixing coefficients.

This chapter is an attempt at providing an abstract unifying framework
that generalizes and sharpens some of the above results.
This framework combines classical martingale techniques
with the
method of \textit{Wasserstein matrices} \cite{follmer79_tail}. In particular, we rely on Wasserstein matrices to obtain general-purpose quantitative estimates of the local variability of a function of many dependent random variables after taking a conditional expectation with respect to a subset of the variables. A concentration inequality in a metric space must
necessarily
capture the
interplay between the metric and the distribution,
and, in our setting, Wasserstein matrices provide the ideal analytical tool for this task.
As an illustration, we recover (and, in certain cases, sharpen) some
results of \cite{kulske03,chazottes07,kontram06} by demonstrating all of these to be
special cases of the Wasserstein matrix method.

The remainder of the chapter is organized as follows. Section~\ref{sec:prelims} is devoted to setting up the basic notation and preliminary definitions. A brief discussion of the concentration of measure phenomenon in high-dimensional spaces is presented in Section~\ref{sec:concentration}, together with a summary of key methods to establish concentration under the independence assumption. Next, in Section~\ref{sec:martingale}, we present our abstract martingale technique and then demonstrate its wide scope in Section~\ref{sec:using_martingales} by deriving many of previously published concentration inequalities as special cases. We conclude in Section~\ref{sec:open} by listing some open questions.

\section{Preliminaries and notation}
\label{sec:prelims}

\subsection{Metric probability spaces} A \textit{metric probability space} is a triple $(\Omega,\mu,d)$, where $\Omega$ is a Polish space equipped with its Borel $\sigma$-field, $\mu$ is a Borel probability measure on $\Omega$, and $d$ is metric on $\Omega$, assumed to be
 a measurable function on the product space $\Omega \times \Omega$. We do not assume that $d$ is the same metric that metrizes the Polish topology on $\Omega$. 

\subsection{Product spaces}
\label{ssec:product_spaces}

Since concentration of measure is a high-dimensional phenomenon, a natural setting for studying it is that of a product space. Let $T$ be a finite index set, which we identify with the set $[n] \deq \{1,\ldots,n\}$, where $n = |T|$ (this amounts to fixing some linear ordering of the elements of $T$). We will use the following notation for subintervals of $T$: $[i] \deq \{1,\ldots,i\}$; $[i,j] \deq \{i,i+1,\ldots,j\}$ for $i \neq j$; $(i,j] \deq \{i+1,\ldots,j\}$ for $i < j$; $(i,j) \deq \{i+1,\ldots,j-1\}$ for $i < j-1$; etc.

With each $i \in T$, we associate a measurable space $(\sX_i,\cB_i)$, where $\sX_i$ is a Polish space and $\cB_i$ is its Borel $\sigma$-field. For each $I \subseteq T$, we will equip the product space $\sX^I \deq \prod_{i \in I}\sX_i$ with the product $\sigma$-field $\cB^I \deq \bigotimes_{i \in I}\cB_i$. When $I = T$, we will simply write $\sX$ and $\cB$. We will write $x^I$ and $x$ for a generic element of $\sX^I$ and $\sX$, respectively. Given two sets $I,J \subset T$ with $I \cap J = \varnothing$, the \textit{concatenation} of  $x^I \in \sX^I$ and $z^J \in \sX^J$ is defined as $y = x^Iz^J \in \sX^{I \cup J}$ by setting
\begin{align*}
	y_i = \begin{cases}
	x_i, & i \in I \\
	z_i, & i \in J
	\end{cases}.
\end{align*}
Given a random object $X = (X_i)_{i \in T}$ taking values in $\sX$ according to a probability law $\mu$, we will denote by $\PP_\mu[\cdot]$ and $\E_\mu[\cdot]$ the probability and expectation with respect to $\mu$, by $\mu^I(\d x^I | x^{J})$ the regular conditional probability law of $X^I$ given $X^{J} = x^{J}$, and by $\mu^I(\d x^I)$ the marginal probability law of $X^I$. When $I = \{i\}$, we will write $\mu_i(\cdot)$ and $\mu_i(\cdot|x^J)$.

For each $i \in T$, we fix a metric on $\sX_i$, which is assumed to be measurable with respect to the product $\sigma$-field $\cB_i \otimes \cB_i$. For each $I \subseteq T$, equip $\sX^I$ with the product metric $\rho^I$, where
\begin{align*}
	\rho^I(x^I,z^I) \deq \sum_{i \in I}\rho_i(x_i,z_i), \qquad \forall x^I,z^I \in \sX^I.
\end{align*}
When $I \equiv T$, we will simply write $\rho$ instead of $\rho^T$. In this way, for any Borel probability measure $\mu$ on $\sX$, we can introduce a ``global'' metric probability space $(\sX,\mu,\rho)$, as well as ``local" metric probability spaces $(\sX^I,\mu^I,\rho^I)$ and $(\sX^I,\mu^I(\cdot|x^J),\rho_I)$ for all $I,J \subset T$ and all $x^J \in \sX^J$. 

\subsection{Couplings and transportation distances} Let $\Omega$ be a Polish space. A \textit{coupling} of two Borel probability measures $\mu$ and $\nu$ on 
$\Omega$ is a Borel probability measure ${\bf P}$ on the product space $\Omega \times \Omega$, such that ${\bf P}(\cdot \times \Omega) = \mu$ and ${\bf P}(\Omega \times \cdot) = \nu$. We denote the set of all couplings of $\mu$ and $\nu$ by $\cC(\mu,\nu)$. Let $d$ be a lower-semicontinuous metric on $\Omega$. We denote by ${\rm Lip}(\Omega,d)$  the space of all functions $\Omega \to \R$ that are Lipschitz with respect to $d$, and by ${\rm Lip}_c(\Omega,d)$ the subset of ${\rm Lip}(\Omega,d)$ consisting of $c$-Lipschitz functions. The \textit{$L^1$ Wasserstein} (or \textit{transportation}) \textit{distance} between $\mu$ and $\nu$ is defined as
\begin{align}
	W_d(\mu,\nu) \deq \inf_{{\bf P} \in \cC(\mu,\nu)} \E_{{\bf P}}[d(X,Y)],
\end{align}
where $(X,Y)$ is a random element of $\Omega \times \Omega$ with law ${\bf P}$. The transportation distance admits a dual (Kantorovich--Rubinstein) representation
\begin{align}\label{eq:Wasserstein_dual}
	W_d(\mu,\nu) = \sup_{f \in {\rm Lip}_1(\Omega,d)} \left| \int_\Omega f \d\mu - \int_\Omega f \d\nu\right|.
\end{align}
For example, when we equip $\Omega$ with the trivial metric $d(\omega,\omega') = {\bf 1}\{\omega \neq \omega'\}$, the corresponding Wasserstein distance coincides with the total variation distance:
\begin{align*}
	W_d(\mu,\nu) = \tvnrm{\mu-\nu} = \sup_{A} |\mu(A) - \nu(A)|,
\end{align*}
where the supremum is over all Borel subsets of $\Omega$.

In the context of the product space $(\sX,\rho)$ defined earlier, we will use the shorthand $W_i$ for $W_{\rho_i}$, $W^I$ for $W_{\rho^I}$, and $W$ for $W_\rho$.

\subsection{Markov kernels and Wasserstein matrices}

 A \textit{Markov kernel} on $\sX$ is a mapping $K : \sX \times \cB \to [0,1]$, such that $x \mapsto K(x,A)$ is measurable for each $A \in \cB$, and $K(x,\cdot)$ is a Borel probability measure on $\sX$ for each $x \in \sX$. Given a Markov kernel $K$ and a bounded measurable function $f : \sX \to \Reals$, we denote by $Kf$ the bounded measurable function
\begin{align*}
	Kf(x) \deq \int_\sX f(y) K(x,\d y), \qquad x \in \sX.
\end{align*}
Likewise, given a Borel probability measure $\mu$ on $\sX$, we denote by $\mu K$ the Borel probability measure
\begin{align*}
	\mu K(A) \deq \int_\sX K(x,A) \mu(\d x), \qquad A \in \cB.
\end{align*}
It is not hard to see that $\int_\sX f \d (\mu K) = \int_\sX (K f)\d\mu$. 

Given a measurable function $f : \sX \to \Reals$, we define the \textit{local oscillation of $f$ at $i \in T$} as
\begin{align*}
  \delta_i(f)
  \deq \sup_{
    x, z
    \in \sX \atop x^{T \bs \{i\}} = z^{T \bs \{i\}}
  }
  \frac{|f(x)-f(z)|}{\rho_i(x_i,z_i)},
\end{align*}
where we follow the convention $0/0 = 0$. This quantity measures the variability of $f$ in its $i$th argument when all other arguments are held fixed. As will become evident later on, our martingale technique for establishing concentration inequalities for a given function $f : \sX \to \Reals$ requires controlling the local oscillations $\delta_i(Kf)$ in terms of the local oscillations $\delta_i(f)$ for appropriately chosen Markov kernels $K$. 

To get an idea of what is involved, let us consider the simple case when each $\sX_i$ is endowed with the scaled trivial metric $\rho_i(x_i,z_i) \deq \alpha_i \1\{x_i \neq z_i\}$, where $\alpha_i > 0$ is some fixed constant. Then
\begin{align*}
  \delta_i(f) = \frac{1}{\alpha_i}\sup\Big\{ |f(x)-f(z)| : x,z \in \sX,\, x^{T\backslash\set{i}} = z^{T\backslash\set{i}}
  \Big\}.
\end{align*}
The corresponding metric $\rho$ on $\sX$ is the weighted Hamming metric
\begin{align}
	 \label{weighted-ham-def}
\rho_{\boldsymbol{\alpha}}(x,z) \deq \sum_{i \in T}\alpha_i \1\{x_i \neq z_i\}.
\end{align}
Fix a Markov kernel $K$ on $\sX$. The \textit{Dobrushin contraction coefficient} 
of $K$ (also associated in the literature with Doeblin's name) is the smallest $\theta\ge0$ for which $\tvnrm{K(x,\cdot)-K(z,\cdot)}\le\theta$
holds for all $x,z\in\sX$. The term \textit{contraction} is justified by the well-known inequality
(apparently going back to Markov himself \cite[$\mathsection$5]{mar1906})
\begin{align}
  \label{eq:doblin}
  \tvnrm{\mu K-\nu K} \le \theta\tvnrm{\mu-\nu},
\end{align}
which holds for all probability measures $\mu,\nu$ on $\sX$. Then we have the following estimate:

\begin{proposition}
  If $K$ is a Markov kernel on $\sX$ with Dobrushin coefficient $\theta$,
  then
  for every $i \in T$ and for every
  $f \in {\rm Lip}(\sX,\rho_{\boldsymbol{\alpha}})$, we have
\begin{align}\label{eq:Dobrushin_bound}
	\delta_i(Kf)\le \frac{\theta}{\alpha_i}\sum_{j \in T}\alpha_j\delta_j(f).
\end{align}
\end{proposition}

\begin{proof} Fix an index $i \in T$ and any two $x,z \in \sX$ that differ only in the $i$th coordinate: $x^{T \bs \{i\}} = z^{T \bs \{i\}}$ and $x_i \neq z_i$. Pick an arbitrary coupling ${\bf P}_{x,z} \in \cC(K(x,\cdot),K(z,\cdot))$. Then
	\begin{align*}
		|Kf(x)-Kf(z)| &= \left|\int_\sX K(x,\d u) f(u) - \int_\sX K(z, \d y) f(y)\right| \\
		&= \left|\int_{\sX
                  \times\sX}
                  {\bf P}_{x,z}(\d u, \d y) \big(f(u)-f(y)\big)\right| \\
		  &\le \sum_{j \in T} \delta_j(f)
                  \int_{\sX\times\sX}
                      {\bf P}_{x,z}(\d u, \d y)
                      \rho_j
                      (u_j,y_j) \\
		    &= \sum_{j \in T} \alpha_j \delta_j(f)
                    \int_{\sX\times\sX}
                        {\bf P}_{x,z}(\d u, \d y) \1{\{u_j \neq y_j\}} \\
		        &\le \sum_{j \in T} \alpha_j \delta_j(f) \cdot
                        \int_{\sX\times\sX}
                            {\bf P}_{x,z}(\d u, \d y) \1\{u \neq y\},
	\end{align*}
where the first inequality is by the definition of $\delta_i(f)$, while the second one follows from the obvious implication $u_j \neq y_j \Rightarrow u \neq y$.
Taking the infimum of both sides over all couplings ${\bf P}_{x,z} \in \cC(K(x,\cdot),K(z,\cdot))$ yields
\begin{align*}
	|Kf(x)-Kf(z)| &\le \sum_{j \in T}\alpha_j \delta_j(f) \cdot \tvnrm{K(x,\cdot)-K(z,\cdot)} \\
	&\le \theta \sum_{j \in T}\alpha_j
        \delta_j(f)
        .
\end{align*}
Finally, dividing both sides of the above inequality by $\alpha_i$ and taking the supremum over all choices of $x,z$ that differ only in the $i$th coordinate, we obtain \eqref{eq:Dobrushin_bound}.
\qed \end{proof}  

One shortcoming of the above result (which is nontrivial only under the rather strong condition
\begin{align}\label{eq:gen_Dobrushin_condition}
\theta < \frac{\alpha_i}{\alpha_j} < \theta^{-1}
\end{align}
for all $i,j \in T$) is that it gives only a very rough idea of the influence of $\delta_j(f)$ for $j \in T$ on $\delta_i(Kf)$. For example, if $\alpha_1 = \ldots = \alpha_n = 1$, then the condition \eqref{eq:gen_Dobrushin_condition} reduces to the Dobrushin contraction condition $\theta < 1$, and the inequality \eqref{eq:Dobrushin_bound} becomes
$$
\delta_i(Kf)
\le \theta \sum_{j \in T} \delta_j(f),
$$
suggesting that all of the $\delta_j(f)$'s influence $\delta_i(Kf)$ equally. However, this picture can be refined. To that end, we introduce the notion of a \textit{Wasserstein matrix} following F\"ollmer \cite{follmer79_tail}. Let us denote by $\bdelta(f)$ the vector $(\delta_i(f))_{i \in T}$. We say that a nonnegative matrix $V = (V_{ij})_{i,j \in T}$ is a Wasserstein matrix for $K$ if, for every $f \in {\rm Lip}(\sX,\rho)$ and for every $i \in T$,
\begin{align}
  \label{eq:deltaKf}
	\delta_i(K f) \le \sum_{j \in T} V_{ij}\delta_j(f),
\end{align}
or, in vector form, if $\bdelta (Kf) \preceq V\bdelta(f)$. 

One of our main objectives will be to show that concentration inequalities for functions $f$ of $X \sim \mu$ can be obtained using Wasserstein matrices for certain Markov kernels $K$ related to $\mu$. In order to motivate the introduction of Wasserstein matrices, we record a couple of contraction estimates for Markov kernels that may be of independent interest. To that end, we introduce another coupling-based distance between probability measures \cite[Chap.~8]{Boucheron-book-2013}: for two Borel probability measures on $\sX$, define
\begin{align}
	\bar{W}(\mu,\nu) \deq \inf_{{\bf P} \in \cC(\mu,\nu)} \sqrt{\sum_{i \in T} \left(\E_{{\bf P}}[\rho_i(X_i,Y_i)]\right)^2},
\end{align}
where $(X,Y)$ is a random element of $\sX \times \sX$. Even though $\bar{W}$ is not a Wasserstein distance, we can use the inequality $\sqrt{a+b} \le \sqrt{a}+\sqrt{b}$ for $a,b \ge 0$ to show that $\bar{W}(\mu,\nu) \le W(\mu,\nu)$.
\begin{proposition}\label{prop:Wasserstein_contraction_bound} Let $V$ be a Wasserstein matrix for a Markov kernel $K$ on $\sX$. Then for any Lipschitz function $f : \sX \to \Reals$,
	\begin{align}\label{eq:Wasserstein_contraction_bound}
		\left| \E_{\mu K}[f(X)] - \E_{\nu K}[f(X)] \right| \le \left\| V \bdelta(f) \right\|_{\ell^2(T)} \bar{W}(\mu,\nu).
	\end{align}
\end{proposition}
\begin{proof}
  Fix an arbitrary coupling ${\bf P} \in \cC(\mu,\nu)$ and let $(X,Y)$ be a random element of $\sX \times \sX$ with law ${\bf P}$. Then
	\begin{align*}
		\left|\E_{\mu K}[f(X)] - \E_{\nu K}[f(X)]\right| &= \left|\E_\mu[Kf(X)] - \E_\nu[Kf(X)]\right| \\
		&= \left|\E_{{\bf P}} \left[Kf(X)-Kf(Y)\right]\right| \\
		& \le \sum_{i \in T} \delta_i(K f) \cdot \E_{{\bf P}}[\rho_i(X_i,Y_i)] \\
		& \le \sum_{i \in T} \sum_{j \in T} V_{ij} \delta_j(f) \cdot \E_{{\bf P}}[\rho_i(X_i,Y_i)].
	\end{align*}
	where in the last step we have used the definition of the Wasserstein matrix. Using the Cauchy--Schwarz inequality, we obtain
	\begin{align*}
		\left|\E_{\mu K}[f(X)] - \E_{\nu K}[f(X)]\right| &\le \sqrt{\sum_{i \in T} \Big|\sum_{j \in T}V_{ij}\delta_j(f)\Big|^2 \cdot \sum_{i \in T} \left(\E_{\bf P}[\rho_i(X_i,Y_i)]\right)^2} \\
		&= \left\| V\bdelta(f) \right\|_{\ell^2(T)} \cdot \sqrt{\sum_{i \in T}\left(\E_{\bf P}[\rho_i(X_i,Y_i)]\right)^2}.
	\end{align*}
Taking the infimum of both sides over all ${\bf P} \in \cC(\mu,\nu)$, we obtain \eqref{eq:Wasserstein_contraction_bound}.
\qed \end{proof}

\begin{corollary} Let $V$ be a Wasserstein matrix for a Markov kernel $K$ on $\sX$. Then, for any two Borel probability measures $\mu$ and $\nu$ on $\sX$,
	\begin{align}
		W(\mu K,\nu K) \le \| V\1 \|_{\ell^2(T)} \bar{W}(\mu,\nu),
	\end{align}
	where $\1 \in \Reals^T$ is the vector of all ones, and therefore
	\begin{align*}
		W(\mu K, \nu K) \le \| V \1 \|_{\ell^2(T)} W(\mu,\nu).
	\end{align*}
\end{corollary}
\begin{proof} A function $f : \sX \to \Reals$ belongs to ${\rm Lip}_1(\sX,\rho)$ if and only if $ \bdelta(f) \in [0,1]^T$. Using the dual representation \eqref{eq:Wasserstein_dual} of $W$ and applying Proposition~\ref{prop:Wasserstein_contraction_bound}, we can write
	\begin{align*}
		W(\mu K, \nu K) &= \sup_{f \in {\rm Lip}_1(\sX,\rho)} \left| \E_{\mu K}[f(X)] - \E_{\nu K}[f(X)]\right| \\
		&\le \sup_{\xi \in [0,1]^T} \| V \xi \|_{\ell^2(T)} \bar{W}(\mu,\nu).
	\end{align*}
Since $V$ is a nonnegative matrix, the supremum is achieved by $\xi = \1$.
\qed \end{proof}

\subsection{Relative entropy} Finally, we will need some key notions from information theory. The \textit{relative entropy} (or \textit{information divergence}) between two probability measures $\mu,\nu$ on a space $\Omega$ is defined as
\begin{align*}
	D(\nu \| \mu) \deq \begin{cases}
	\displaystyle \int_\Omega \d\mu\, f \log f, & \text{if $\nu \ll \mu$ with $f = \d\nu/\d\mu$} \\
	+ \infty, & \text{otherwise}
\end{cases}.
\end{align*}
We use natural logarithms throughout the chapter. The relative entropy is related to the total variation distance via Pinsker's inequality\footnote{
  Though commonly referred to as {\it Pinsker's inequality},
  (\ref{eq:Pinsker}) as given here (with the optimal constant $\frac12$)
  was proven by Csisz{\'a}r \cite{MR0219345}
  and Kullback
\cite{kullback67}
in 1967.
  }
\begin{align}\label{eq:Pinsker}
	\tvnrm{\mu-\nu} \le \sqrt{\frac{1}{2}D(\mu \| \nu)}.
\end{align}

\section{Concentration of measure and sufficient conditions}
\label{sec:concentration}

In this section, we give a precise definition of the concentration of measure phenomenon, review several sufficient conditions for it to hold, and briefly discuss how it can be established under the independence assumption via tensorization. For more details and further references, the reader can consult \cite{Boucheron-book-2013} or \cite{Raginsky_Sason_FnT}.

We say that the metric probability space $(\sX,\mu,\rho)$ has the concentration of measure property if there exists a positive constant $c > 0$, such that, for every Lipschitz function $f : \sX \to \R$,
\begin{align}\label{eq:COM}
	\PP_\mu\left\{ f(X) - \E_\mu[f(X)] \ge t \right\} \le e^{-t^2/2c\|f\|^2_{\rm Lip}}, \qquad \forall t > 0
\end{align}
where
\begin{align*}
  \| f \|_{\rm Lip} \deq \sup_{
    x,y
    \in \sX \atop x \neq y} \frac{|f(x)-f(y)|}{\rho(x,y)}
\end{align*}
is the Lipschitz constant of $f$. A sufficient (and, up to constants, necessary)
condition for \eqref{eq:COM} is that, for every $f \in {\rm Lip}_1(\sX,\rho)$, the random variable $f(X)$ with  $X \sim \mu$ is \textit{$c$-subgaussian}, i.e.,
\begin{align}\label{eq:SG}
	\log \E_\mu \left[e^{\lambda(f(X)-\E_\mu[f(X)])}\right] \le \frac{c\lambda^2}{2}, \qquad \forall \lambda \in \R.
\end{align}
A fundamental result of Bobkov and G\"otze \cite{Bobkov_Goetze} states that the subgaussian estimate \eqref{eq:SG} holds for all $f \in {\rm Lip}_1(\sX,\rho)$ if and only if $\mu$ satisfies the so-called \textit{transportation-information inequality}
\begin{align}\label{eq:TC}
	W(\mu,\nu) \le \sqrt{2c\, D(\nu \| \mu)},
\end{align}
where $\nu$ ranges over all Borel probability measures on $\sX$. We will use the shorthand $\mu \in T_\rho(c)$ to denote the fact that the inequality \eqref{eq:TC} holds for all $\nu$. The key role of transportation-information inequalities in characterizing the concentration of measure phenomenon was first recognized by Marton in a  breakthrough paper \cite{marton96}, with further developments in \cite{marton98,marton03,marton04}.

The \textit{entropy method} (see, e.g., \cite[Chap.~6]{Boucheron-book-2013} and \cite[Chap.~3]{Raginsky_Sason_FnT}) provides another route to establishing \eqref{eq:SG}. Its underlying idea can be briefly described as follows. Given a measurable function $f : \sX \to \R$, consider the logarithmic moment-generating function
$$
\psi_f(\lambda) \deq \log \E_\mu\left[e^{\lambda(f(X)-\E_\mu[f(X)])}\right]
$$
of the centered random variable $f(X) - \E_\mu[f(X)]$. For any $\lambda \neq 0$, introduce the \textit{tilted probability measure} $\mu^{(\lambda f)}$ with
$$
\frac{\d\mu^{(\lambda f)}}{\d\mu} = \frac{e^{\lambda f}}{\E_\mu[e^{\lambda f}]} = \frac{e^{\lambda (f - \E_\mu f)}}{e^{\psi_f(\lambda)}}.
$$
Then a simple calculation shows that the relative entropy $D(\mu^{(\lambda f)} \| \mu)$ can be expressed as
\begin{align*}
	D(\mu^{(\lambda f)} \| \mu) = \lambda \psi'_f(\lambda) - \psi_f(\lambda) \equiv \lambda^2 \left(\frac{\psi_f(\lambda)}{\lambda}\right)'
\end{align*}
where the prime denotes differentiation with respect to $\lambda$. Using the fact that $\psi_f(0) = 0$ and integrating, we obtain the following formula for $\psi_f(\lambda)$:
\begin{align}\label{eq:psi_entropy}
\psi_f(\lambda) = \lambda \int^\lambda_0 \frac{D(\mu^{(t f)}\| \mu)}{t^2} \d t.
\end{align}
This representation is at the basis of the so-called \textit{Herbst argument}, which for our purposes can be summarized as follows:
\begin{lemma}[Herbst] The metric probability space $(\sX,\mu,\rho)$ has the concentration property with constant $c$ if, for any $f \in {\rm Lip}_1(\sX,\rho)$, 
	\begin{align}\label{eq:Herbst}
	D(\mu^{(t f)} \| \mu) \le \frac{ct^2}{2}, \qquad \forall t > 0.
	\end{align}
\end{lemma}
\begin{remark}
  {Up to a constant, the converse is also true \cite[Prob.~3.12]{Ramon_lecture_notes}: if the subgaussian estimate \eqref{eq:SG} holds for every $f \in {\rm Lip}_1(\sX,\rho)$, then 
	\begin{align*}
		D(\mu^{(t f)} \| \mu) \le 2ct^2, \qquad \forall t > 0
	\end{align*}
for every $f \in {\rm Lip}_1(\sX,\rho)$.}
\end{remark}
In this way, the problem of establishing the concentration phenomenon reduces to showing that \eqref{eq:Herbst} holds for every $f \in {\rm Lip}_1(\sX,\rho)$, typically via logarithmic Sobolev inequalities or other functional inequalities.

\subsection{Concentration of measure under the independence assumption}

To set the stage for the general treatment of the concentration phenomenon in high dimensions, we first consider the independent case, i.e., when coordinates $X_i$, $i \in T$, of the random object $X \sim \mu$ are mutually independent. In other words, the probability measure $\mu$ is equal to the product of its marginals: $\mu = \mu_1 \otimes \ldots \otimes \mu_n$. The key to establishing the concentration property in such a setting is \textit{tensorization}, which is an umbrella term for any result that allows one to derive the ``global" concentration property of the high-dimensional product space $(\sX_1 \otimes \ldots \otimes \sX_n, \rho, \mu_1 \otimes \ldots \otimes \mu_n)$ from ``local'' concentration properties of the coordinate spaces $(\sX_i,\rho_i,\mu_i)$, $i \in T$.

Below, we list two such tensorization results, one for the transportation-information inequalities and one for the relative entropy. Both of these results are deep consequences of the interplay between the independence structure of $\mu$ and the metric structure of $\rho$. Indeed, a function $f : \sX \to \R$ belongs to ${\rm Lip}_1(\sX,\rho)$ if and only if $\delta_i(f) \le 1$ for all $i \in T$, i.e., if and only if, for every $i \in T$ and every $x^{T\backslash\{i\}} \in \sX^{T \backslash \{i\}}$, the function $f_i : \sX_i \to R$ given by $f_i(y_i) \deq f(y_ix^{T\backslash\{i\}})$ is $1$-Lipschitz with respect to the metric $\rho_i$ on $\sX_i$. With this in mind, it is reasonable to expect that if one can establish a concentration property for all $1$-Lipschitz functions on the coordinate spaces $\sX_i$, then one can deduce a concentration property for functions on the product space $\sX$ that are $1$-Lipschitz in each coordinate.

\begin{lemma}[Tensorization of transportation-information inequalities] Suppose that there exist constants $c_1,\ldots,c_n \ge 0$, such that
	$$
	\mu_i \in T_{\rho_i}(c_i), \qquad \forall i \in T.
$$ Then $\mu = \mu_1 \otimes \ldots \otimes \mu_n \in T_\rho(c)$ with $c = \sum^n_{i=1}c_i$.
\end{lemma}
For example, by an appropriate rescaling of Pinsker's inequality \eqref{eq:Pinsker}, we see that, if each coordinate space $\sX_i$ is endowed with the scaled trivial metric $\rho_i(x_i,z_i) = \alpha_i {\bf 1}\{x_i \neq z_i\}$ for some $\alpha_i > 0$, then any Borel probability measure $\mu_i$ on $\sX_i$ satisfies the transportation-information inequality with $c_i = \alpha^2_i/4$.  By the above tensorization lemma, any product measure $\mu_1 \otimes \ldots \otimes \mu_n$ on the product space $\sX_1 \otimes \ldots \otimes \sX_n$ equipped with the weighted Hamming metric $\rho_{\boldsymbol{\alpha}}$ defined in \eqref{weighted-ham-def}
satisfies $T_{\rho_{\boldsymbol{\alpha}}}(c)$ with $c = \frac{1}{4}\sum_{i \in T}\alpha^2_i$. Consequently, by the Bobkov--G\"otze theorem, the subgaussian estimate \eqref{eq:SG} holds for any function $f \in {\rm Lip}_1(\sX,\rho_{\boldsymbol{\alpha}})$, which in turn implies, via the Chernoff bound, that
$$
\PP_\mu\Big\{ f - \E_\mu f \ge t \Big\} \le \exp\left(-\frac{2t^2}{\sum_{i \in T}\alpha^2_i}\right), \qquad \forall t \ge 0.
$$
This provides an alternative derivation of McDiarmid's inequality (with the sharp constant in the exponent), which was originally proved using the martingale method.

\begin{lemma}[Tensorization of
    relative
    entropy] Consider a product measure $\mu = \mu_1 \otimes \ldots \otimes \mu_n$. Then for any other probability measure $\nu$ on $\sX$ we have
	\begin{align}
		D(\nu \| \mu) \le \sum_{i \in T} \E_\nu D\Big{(}\nu_i(\cdot|X^{T\backslash\{i\}}) \| \mu_i\Big{)}.
	\end{align}
\end{lemma}
The idea is to apply this lemma to $\nu = \mu^{(t f)}$ for some $t \ge 0$ and an arbitrary $f \in {\rm Lip}_1(\sX,\rho)$. In that case, a simple calculation shows that the conditional probability measure
$\nu_i(\d x_i |x^{T \backslash \{i\}}) = \mu^{(tf)}_i(\d x_i|x^{T\backslash\{i\}})$ is equal to the tilted distribution $\mu^{(tf_i)}_i$ with $f_i(x_i) = f(x_ix^{T\backslash \{i\}})$, and therefore
$$ 
D(\mu^{(tf)} \| \mu) \le \sum^n_{i=1} \E_{\mu^{(tf)}} D\big(\mu^{(t f_i)}_i \big\| \mu_i \big).
$$
If $f \in {\rm Lip}_1(\sX,\rho)$, then $f_i \in {\rm Lip}_1(\sX_i,\rho_i)$. Thus, if we can show that, for any $g \in {\rm Lip}_1(\sX_i,\rho_i)$, 
$$
D(\mu^{(t g)}_i \| \mu_i) \le \frac{c_i t^2}{2}, \qquad \forall t \ge 0,
$$
then the estimate
$$
D(\mu^{(tf)} \| \mu) \le \frac{ct^2}{2}, \qquad \forall t \ge 0
$$
holds with $c = \sum^n_{i=1}c_i$ for all $f \in {\rm Lip}(\sX,\rho)$ by the tensorization lemma. Invoking the Herbst argument, we conclude that $(\sX,\mu,\rho)$ has the concentration property with the same $c$.

\section{The abstract martingale method}
\label{sec:martingale}

In this section, we present a general martingale-based scheme for deriving concentration inequalities for functions of many dependent random variables. Let $f : \sX \to \Reals$ be the function of interest, and let $X = (X_i)_{i \in T}$ be a random element of the product space $\sX$ with probability law $\mu$. Let $\cF_0 \subset \cF_1 \subset \ldots \subset \cF_m$ be a filtration (i.e., an increasing sequence of $\sigma$-fields) on $\sX$, such that $\cF_0$ is trivial and $\cF_m = \cB$. The idea is to decompose the centered random variable $f(X)-\E_\mu[f(X)]$ as a sum of martingale differences
$$
M^{(j)} \deq \E_\mu[f(X)|\cF_j] - \E_\mu[f(X)|\cF_{j-1}], \qquad j = 1,\ldots,m.
$$
By construction, $\E_\mu[f(X)|\cF_m] = f(X)$ and $\E_\mu[f(X)|\cF_0] = \E_\mu[f(X)]$, so the problem of bounding the probability $\PP_\mu\big\{|f-\E_\mu f| \ge t\big\}$ for a given $t \ge 0$ reduces to bounding the probability
$$
 \PP_\mu\Bigg\{ \Big|\sum^m_{j=1} M^{(j)}\Big| \ge t\Bigg\}.
$$
The latter problem hinges on being able to control the martingale differences $M^{(j)}$. In particular, if each $M^{(j)}$ is a.s.\ bounded, we have the following:

\begin{theorem}[Azuma--Hoeffding inequality] Let $\{M^{(j)}\}^m_{j=1}$ be a martingale difference sequence with respect to a filtration $\{\cF_j\}^m_{j=0}$. Suppose that, for each $j$, there exist $\cF_{j-1}$-measurable random variables $A^{(j)}$ and $B^{(j)}$, such that $A^{(j)} \le M^{(j)} \le B^{(j)}$ a.s. Then 
	\begin{align}\label{eq:exponential_AH}
		\E\left[\exp\left(\lambda\sum^m_{j=1}M^{(j)}\right)\right] \le \exp\left(\frac{\lambda^2\sum^m_{j=1}\| B^{(j)}-A^{(j)}\|^2_\infty}{8}\right), \quad \forall \lambda \in \R.
	\end{align}
Consequently, for any $t \ge 0$,
	\begin{align}\label{eq:AH}
		\PP\Bigg\{\Big| \sum^m_{j=1} M^{(j)}\Big| \ge t\Bigg\} \le 2 \exp\left(-\frac{2t^2}{\sum^m_{j=1} \|B^{(j)}-A^{(j)}\|^2_\infty}\right).
	\end{align}
\end{theorem}

The most straightforward choice of the filtration is also the most natural one: take $m = |T| = n$, and for each $i \in T$ take $\cF_i = \sigma(X^{[i]})$. For $i \in T$, define a Markov kernel $K^{(i)}$ on $\sX$ by
\begin{align}\label{eq:Ki_def}
	K^{(i)}(x, \d y) \deq \delta_{x^{[i-1]}}(\d y^{[i-1]}) \otimes \mu^{[i,n]}(\d y^{[i,n]} |x^{[i-1]}).
\end{align}
Then, for any $f \in L^1(\mu)$ we have
\begin{align*}
	K^{(i)} f(x) &= \int_\sX f(y) K^{(i)}(x, \d y) \\
	&= \int_{\sX^{[i,n]}} f(x^{[i-1]}y^{[i,n]}) \mu^{[i,n]}(\d y^{[i,n]} | x^{[i-1]}) \\
	&= \E_\mu[f(X)|X^{[i-1]}=x^{[i-1]}];
\end{align*}
in particular, $K^{(1)} f = \E_\mu f$. We extend this definition to $i = n+1$ in the obvious way:
\begin{align*}
K^{(n+1)}(x,\d y) = \delta_{x}(\d y),
\end{align*}
so that $K^{(n+1)}f = f$. Then, for each $i \in T$, we can write $M^{(i)} = K^{(i+1)}f - K^{(i)}f$. With this construction, we can state the following theorem that applies to the case when each coordinate space $\sX_i$ is endowed with a bounded measurable metric $\rho_i$:

\begin{theorem}\label{thm:mart_1} Assume that, for all $i$,
\begin{align*}
\| \rho_i \| \deq \sup_{x_i,z_i \in \sX_i} \rho_i(x_i,z_i) < \infty.
\end{align*} For each $i \in \{1,\ldots,n+1\}$, let $V^{(i)}$ be a
Wasserstein matrix for the Markov kernel $K^{(i)}$ defined in \eqref{eq:Ki_def},
in the sense that $\bdelta (K^{(i)}f) \preceq V^{(i)}\bdelta(f)$
holds for each $f\in {\rm Lip}(\sX,\rho)$
as in (\ref{eq:deltaKf}).
Define the matrix $\Gamma = (\Gamma_{ij})_{i,j \in T}$ with entries
	\begin{align*}
		\Gamma_{ij} \deq \| \rho_i \| V^{(i+1)}_{ij}.
	\end{align*}
        Then, for any $f \in {\rm Lip}(\sX,\rho)$
        and  for any $t \ge 0$, we have
\begin{align}\label{eq:martingale_bound}
	\PP_\mu\Big\{ |f(X) - \E_\mu[f(X)]| \ge t \Big\} \le 2 \exp\left(-\frac{2t^2}{\| \Gamma \bdelta(f) \|^2_{\ell^2(T)}}\right).
\end{align}
\end{theorem}
\begin{proof} For each $i \in T$, using the tower property of conditional expectations, we can write
	\begin{align*}
		M^{(i)} &= \E_\mu[f(X)|X^{[i]} = x^{[i]}] - \E_\mu[f(X)|X^{[i-1]} = x^{[i-1]}] \nonumber\\
		&= \E_\mu[f(X)|X^{[i]}=x^{[i]}] - \E_\mu\big[\E_\mu[f(X)|X^{[i-1]}=x^{[i-1]},X_i]\big|X^{[i-1]}=x^{[i-1]}\big] \nonumber\\
		&= \int_{\sX^{[i,n]}} \mu^{[i,n]}(\d y^{[i,n]}|x^{[i-1]}) \Big(\int_{\sX^{(i,n]}} \mu^{(i,n]}(\d z^{(i,n]}|x^{[i]})f(x^{[i-1]}x_iz^{(i,n]}) \nonumber\\
		& \qquad \qquad \qquad- \int_{\sX^{(i,n]}} \mu^{(i,n]}(\d z^{(i,n]}|x^{[i-1]}y_i) f(x^{[i-1]}y_iz^{(i,n]})\Big) \nonumber\\
		&= \int_{\sX^{[i,n]}} \mu^{[i,n]}(\d y^{[i,n]}|x^{[i-1]})\left(K^{(i+1)}f(x^{[i-1]}x_iy^{(i,n]}) - K^{(i+1)}f(x^{[i-1]}y_iy^{(i,n]})\right).
	\end{align*}
        From this, it follows that $A^{(i)} \le M^{(i)} \le B^{(i)}$ a.s., where 
\begin{align*}
	A^{(i)} &\deq \int_{\sX^{[i,n]}} \mu^{[i,n]}(\d y^{[i,n]}|x^{[i-1]})\inf_{x_i \in \sX_i}\left(K^{(i+1)}f(x^{[i-1]}x_iy^{(i,n]}) - K^{(i+1)}f(x^{[i-1]}y_iy^{(i,n]})\right) \\
	B^{(i)} &\deq \int_{\sX^{[i,n]}} \mu^{[i,n]}(\d y^{[i,n]}|x^{[i-1]})\sup_{x_i \in \sX_i}\left(K^{(i+1)}f(x^{[i-1]}x_iy^{(i,n]}) - K^{(i+1)}f(x^{[i-1]}y_iy^{(i,n]})\right),
\end{align*}
and
\begin{align}\label{eq:cond_range}
	\| B^{(i)} - A^{(i)} \|_\infty \le \| \rho_i \| \delta_i\big(K^{(i+1)}f\big).
\end{align}
By definition of the Wasserstein matrix, we have
	\begin{align*}
		\delta_i\left(K^{(i+1)}f\right) \le \sum_{j \in T} V^{(i+1)}_{ij} \delta_j(f).
	\end{align*}
Substituting this estimate into \eqref{eq:cond_range}, we get
	\begin{align}
		\sum^n_{i=1} \| B^{(i)} - A^{(i)} \|^2_\infty \le \sum^n_{i=1} \left|\left( \Gamma \bdelta(f)\right)_i\right|^2 \equiv \left\| \Gamma \bdelta(f) \right\|^2_{\ell^2(T)}.
	\end{align}
The probability estimate \eqref{eq:martingale_bound} then follows from the Azuma--Hoeffding inequality \eqref{eq:AH}.
\qed \end{proof}

We can also use the martingale method to obtain a tensorization result for transportation inequalities without independence assumptions. This result, which generalizes a theorem of Djellout, Guillin, and Wu \cite[Thm.~2.11]{djellout_guillin_wu}, can be used even when the metrics $\rho_i$ are not necessarily bounded.

\begin{theorem}\label{thm:mart_2} Suppose that there exist constants $c_1,\ldots,c_n \ge 0$, such that
\begin{align}\label{eq:TC_i}
	\mu_i(\cdot|x^{[i-1]}) \in T_{\rho_i}(c_i), \qquad \forall i \in T,\, x^{[i-1]} \in \sX^{[i-1]}.
\end{align}
For each $i \in \{1,\ldots,n+1\}$, let $V^{(i)}$ be a Wasserstein matrix for $K^{(i)}$. Then $\mu \in T_\rho(c)$ with
\begin{align}\label{eq:dep_tensor}
	c = \sum_{i \in T} c_i \Big(\sum_{j \in T} V^{(i+1)}_{ij}\Big)^2.
\end{align}
\end{theorem}
\begin{proof} By the Bobkov--G\"otze theorem \cite{Bobkov_Goetze}, it suffices to show that, for every $f : \sX \to \R$ with $\| f \|_{\rm Lip} \le 1$, the random variable $f(X)$ with $X \sim \mu$ is $c$-subgaussian, with $c$ given by \eqref{eq:dep_tensor}. To that end, we again consider the martingale decomposition
	$$
	f - \E_\mu[f] = \sum_{i \in T} M^{(i)}
	$$
	with $M^{(i)} = K^{(i+1)}f-K^{(i)}f$. We will show that, for every $i$,
	\begin{align}\label{eq:ith_term}
		\log \E_\mu\Big[ e^{\lambda M^{(i)}}\Big|X^{[i-1]}\Big] \le \frac{c_i\Big(\sum_{j \in T}V^{(i+1)}_{ij}\Big)^2\lambda^2}{2}, \qquad \forall \lambda \in \R.
	\end{align}
This, in turn, will yield the desired subgaussian estimate
\begin{align*}
	\E_\mu\left[e^{\lambda (f - \E_\mu[f])}\right] &= \E_\mu\left[\exp\left(\lambda \sum_{i \in T}M^{(i)}\right)\right]\\
	&\le \exp\left(\frac{c\lambda^2}{2}\right)
\end{align*}
for every $\lambda \in \R$.

To proceed, note that, for a fixed realization $x^{[i-1]}$ of $X^{[i-1]}$, $M^{(i)} = K^{(i+1)}f - K^{(i)}f$ is $\sigma(X_i)$-measurable, and
\begin{align*}
	\| M^{(i)} \|_{\rm Lip} &\le \sup_{x,y \in \sX \atop x^{T\bs \{i\}} = y^{T\bs \{i\}}} \frac{\left|K^{(i+1)}f(x)-K^{(i+1)}f(y)\right|}{\rho_i(x_i,y_i)} \\
	&\equiv \delta_i\big(K^{(i+1)}f\big) \\
	&\le \sum_{j \in T} V^{(i+1)}_{ij}\delta_j(f) \\
	&\le \sum_{j \in T} V^{(i+1)}_{ij},
\end{align*}
where we have used the definition of the Wasserstein matrix, as well as the fact that $\| f \|_{\rm Lip} \le 1$ is equivalent to $\delta_j(f) \le 1$ for all $j \in T$. Since $\mu_i(\cdot|x^{[i-1]}) \in T_{\rho_i}(c)$ by hypothesis, we obtain the estimate \eqref{eq:ith_term} by the Bobkov--G\"otze theorem.
\qed \end{proof}

As a sanity check, let us confirm that, in the case when $\mu$ is a product measure and the product space $\sX$ is endowed with the weighted Hamming metric $\rho_{{\boldsymbol{\alpha}}}$ defined in \eqref{weighted-ham-def}, Theorems~\ref{thm:mart_1} and \ref{thm:mart_2} both reduce to McDiarmid's inequality. To see this, we first note that, when the $X_i$'s are independent, we can write
$$
K^{(i)} f (x) = \int_{\sX^{[i,n]}} f(x^{[i-1]}y^{[i,n]}) \mu_i(\d y_i) \mu_{i+1}(\d y_{i+1}) \ldots \mu_n(\d y_n)
$$
for each $i \in T$, $f \in L^1(\mu)$, and $x \in \sX$. This, in turn, implies that
\begin{align*}
\delta_i(K^{(i+1)}f) &= \alpha^{-1}_i\sup_{x,z \in \sX \atop x^{T\backslash\{i\}} = z^{T\backslash\{i\}}} \left| K^{(i+1)}f(x)-K^{(i+1)}f(z)\right| \\
&= \alpha^{-1}_i\sup_{x,z \in \sX \atop x^{T\backslash\{i\}} = z^{T\backslash\{i\}}} \Big|  \int_{\sX^{(i,n]}} f(x^{[i]}y^{(i,n]}) \mu_{i+1}(\d y_{i+1})  \ldots \mu_n(\d y_n)\\
& \qquad \qquad \qquad -  \int_{\sX^{(i,n]}} f(z^{[i]}y^{(i,n]}) \mu_{i+1}(\d y_{i+1}) \ldots \mu_n(\d y_n)\Big| \\
&\le \alpha^{-1}_i \sup_{x,z \in \sX \atop x^{T\backslash\{i\}} = z^{T\backslash\{i\}}} \left|f(x)-f(z)\right| \\
&= \delta_i(f),
\end{align*}
where we have used the fact that, with $\rho_i(x_i,z_i) = \alpha_i \1\{x_i \neq z_i\}$, $\|\rho_i\|=\alpha_i$ for every $i \in T$. Therefore, for each $i \in T$, we can always choose a Wasserstein matrix $V^{(i+1)}$ for $K^{(i+1)}$ in such a way that its $i$th row has zeroes everywhere except for the $i$th column, where it has a $1$. Now, for any function $f : \sX \to \R$ which is $1$-Lipschitz with respect to $\rho_{\boldsymbol{\alpha}}$, we can take $\bdelta(f) = \1$. Therefore, for any such $f$ Theorem~\ref{thm:mart_1} gives 
\begin{align*}
	\PP_\mu\Bigg\{ |f(X)-\E_\mu[f(X)]| \ge t \Bigg\} \le 2\exp\left(-\frac{2t^2}{\sum^n_{i=1}\alpha^2_i}\right), \qquad \forall t \ge 0
\end{align*}
which is precisely McDiarmid's inequality. Since the constant $2$ in McDiarmid's inequality is known to be sharp, this shows that the coefficient $2$ in the exponent in \eqref{eq:martingale_bound} is likewise optimal. Moreover, with our choice of $\rho_i$, condition \eqref{eq:TC_i} of Theorem~\ref{thm:mart_2} holds with $c_i = \alpha^2_i/4$, and, in light of the discussion above, we can arrange $\sum_j V^{(i+1)}_{ij} = 1$. Therefore, by Theorem~\ref{thm:mart_2}, any function $f : \sX \to \R$ which is $1$-Lipschitz with respect to $\rho_{\boldsymbol{\alpha}}$ is $c$-subgaussian with constant
$$
c = \sum_{i \in T} c_i \Bigg(\sum_{j \in T}V^{(i+1)}_{ij}\Bigg)^2 = \frac{1}{4}\sum_{i \in T} \alpha^2_i,
$$
which is just another equivalent statement of McDiarmid's inequality.

It is also possible to consider alternative choices of the filtration $\{\cF_j\}^m_{j=0}$. For example, if we partition the index set $T$ into $m$ disjoint subsets (blocks) $T_1,\ldots,T_m$, we can take
$$
\cF_j \deq \sigma\left(X_i : i \in \Lambda_j\right), \qquad \forall i \in T
$$
where $\Lambda_j \deq T_1 \cup \ldots \cup T_j$. Defining for each $j \in [m]$ the Markov kernel
$$
\tilde{K}^{(j)}(x, \d y) \deq \delta_{x^{\Lambda_{i-1}}}(\d y^{\Lambda_{i-1}}) \otimes \mu^{T\backslash \Lambda_{i-1}}(\d y^{T \backslash \Lambda_{i-1}}|x^{\Lambda_{i-1}}),
$$
we can write
$$
M^{(j)} = \E_\mu[f(X)|\cF_j] - \E_\mu[f(X)|\cF_{j-1}] = K^{(j+1)} f - K^{(j)} f
$$
for every $j \in [m]$. As before, we take $K^{(1)}f = \E_\mu[f]$ and $K^{(m+1)}f = f$. Given a measurable function $f : \sX \to \Reals$, we can define the oscillation of $f$ in the $j$th block $T_j$, $j \in [m]$, by
\begin{align*}
	\tilde{\delta}_j(f) \deq \sup_{x,z \in \sX \atop x^{T\backslash T_j} = z^{T \backslash T_j}} \frac{|f(x)-f(z)|}{\rho^{T_j}(x^{T_j},z^{T_j})}.
\end{align*}
The definition of a Wasserstein matrix is modified accordingly: we say that a nonnegative matrix $\tilde{V}=(\tilde{V}_{jk})_{j,k \in [m]}$ is a Wasserstein matrix for a Markov kernel $K$ on $\sX$ with respect to the partition $\{T_j\}^m_{j=1}$ if, for any Lipschitz function $f : \sX \to \Reals$,
\begin{align*}
	\tilde{\delta}_j(Kf) \le \sum^m_{k=1} \tilde{V}_{jk}\tilde{\delta}_k(f)
\end{align*}
for all $j \in [m]$. With these definitions at hand, the following theorem, which generalizes a result of Paulin \cite[Thm.~2.1]{Paulin15}, can be proved in the same way as Theorem~\ref{thm:mart_1}:
\begin{theorem} For each $j \in [m+1]$, let $\tilde{V}^{(j)} = (\tilde{V}^{(j)}_{k\ell})_{k,\ell \in [m]}$ be a Wasserstein matrix for $\tilde{K}^{(j)}$ with respect to the partition $\{T_j\}$. Define the matrix $\tilde{\Gamma} = (\tilde{\Gamma}_{k\ell})_{k,\ell \in [m]}$ with entries
	\begin{align*}
		\tilde{\Gamma}_{k\ell} \deq \| \rho^{T_k} \| \tilde{V}^{(k+1)}_{k\ell},
	\end{align*}
	where
	\begin{align*}
		\| \rho^{T_k} \| \deq \sup_{x^{T_k},z^{T_k}} \rho^{T_k}(x^{T_k},z^{T_k})
	\end{align*}
	is the diameter of the metric space $(\sX^{T_k},\rho^{T_k})$. Then, for any $f \in {\rm Lip}(\sX,\rho)$ and for any $t \ge 0$,
	\begin{align*}
		\PP_\mu\Big\{ |f(X) - \E_\mu[f(X)]| \ge t \Big\} \le 2 \exp\left(-\frac{2t^2}{\big\| \tilde{\Gamma} \tilde{\bdelta}(f) \big\|^2_{\ell^2(m)}}\right).
	\end{align*}
\end{theorem}

\section{The martingale method in action}
\label{sec:using_martingales}

We now show that several previously published concentration inequalities for functions of dependent random variables arise as special cases of Theorem~\ref{thm:mart_1} by exploiting the freedom to choose the Wasserstein matrices $V^{(i)}$. In fact, careful examination of the statement of Theorem~\ref{thm:mart_1} shows that, for each $i \in T$, we only need to extract the $i$th row of $V^{(i+1)}$.

\subsection{Concentration inequalities under the Dobrushin uniqueness condition}

One particularly clean way of constructing the desired Wasserstein matrices is via the classical comparison theorem of Dobrushin for Gibbs measures \cite{Dobrushin_comparison}. For our purposes, we give its formulation due to F\"ollmer \cite{follmer_comparison}:

\begin{lemma}\label{lm:Dobrushin} Let $\nu$ and $\tilde{\nu}$ be two Borel probability measures on $\sX$. Define the matrix $C^\nu = (C^\nu_{ij})_{i,j \in T}$ and the vector $b^{\nu,\tilde{\nu}} = (b^{\nu,\tilde{\nu}}_i)_{i \in T}$ by
  \begin{align}
    \label{dobrushin-C-def}
    C^\nu_{ij} = \sup_{
      x,z
      \in \sX \atop x^{T \bs \{j\}} = z^{T \bs \{j\}}} \frac{W_{i}\big( \nu_i(\cdot|x^{T \bs \{i\}}), \nu_i(\cdot|z^{T \bs \{i\}}) \big)}{\rho_j(x_j,z_j)}
	\end{align}
	and
	\begin{align}
		b^{\nu,\tilde{\nu}}_i = \int_{\sX^{T \bs \{i\}}}\tilde{\nu}^{T \bs \{i\}}(\d x^{T \bs \{i\}}) W_{i}\big( \nu_i(\cdot|x^{T \bs \{i\}}),\tilde{\nu}_i(\cdot|x^{T \bs \{i\}})\big).
	\end{align}
	Suppose that the spectral radius of $C^\nu$ is strictly smaller than unity.
	Then, for any $f \in L^1(\mu)$,
	\begin{align}
		\left|\E_\nu f - \E_{\tilde{\nu}} f\right| \le \sum_{j,k \in T} \delta_j(f) D^\nu_{jk} b^{\nu,\tilde{\nu}}_k,
	\end{align}
	where $D^\nu \deq \sum^\infty_{m=0} (C^\nu)^m$.
\end{lemma}
\begin{remark} {The matrix $C^\nu$ is called the {\em Dobrushin interdependence matrix} of $\nu$. When the spectral radius of $C^\nu$ is strictly smaller than unity, we say that $\nu$ satisfies the {\em Dobrushin uniqueness condition}. This condition is used in statistical physics to establish the absence of phase transitions, which is equivalent to uniqueness of a global Gibbs measure consistent with a given local specification (see the book of Georgii \cite{Georgii} for details).}
\end{remark}

Given an index $i \in T$, we will extract the $i$th row of a Wasserstein matrix for $K^{(i+1)}$ by applying the Dobrushin comparison theorem to a particular pair of probability measures on $\sX$. Let $x,z \in \sX$ be two configurations that differ only in the $i$th coordinate: $x^{T \bs \{i\}} = z^{T \bs \{i\}}$ and $x_i \neq z_i$. Thus, we can write $z = x^{[i-1]}z_ix^{(i,n]}$, and
\begin{align}
	&K^{(i+1)} f(x) - K^{(i+1)} f(z) \nonumber\\
	&= \E_\mu[f(X)|X^{[i]} = x^{[i-1]}x_i] - \E_\mu[f(X)|X^{[i]}=x^{[i-1]}z_i] \nonumber\\
	&= \int_{\sX^{(i,n]}} f(x^{[i-1]}x_iy^{(i,n]})\mu^{(i,n]}(\d y^{(i,n]}|x^{[i-1]}x_i) - \int_{\sX^{(i,n]}} f(x^{[i-1]}z_iy^{(i,n]})\mu^{(i,n]}(\d y^{(i,n]}|x^{[i-1]}z_i) \nonumber\\
	&= \int_{\sX^{(i,n]}} \left(f(x^{[i-1]}x_iy^{(i,n]}) -  f(x^{[i-1]}z_iy^{(i,n]})\right)\mu^{(i,n]}(\d y^{(i,n]}|x^{[i-1]}x_i) \nonumber\\
	& \qquad \qquad + \int_{\sX^{(i,n]}} f(x^{[i-1]}z_iy^{(i,n]})\mu^{(i,n]}(\d y^{(i,n]}|x^{[i-1]}x_i) \nonumber\\
	& \qquad \qquad - \int_{\sX^{(i,n]}} f(x^{[i-1]}z_iy^{(i,n]})\mu^{(i,n]}(\d y^{(i,n]}|x^{[i-1]}z_i).\label{eq:deltaK_1}
\end{align}
By definition of the local oscillation, the first integral in \eqref{eq:deltaK_1} is bounded by $\delta_i(f) \rho_i(x_i,z_i)$. To handle the remaining terms, define two probability measures $\nu,\tilde{\nu}$ on $\sX$ by
\begin{align*}
	\nu(\d y) &\deq \delta_{x^{[i-1]}z_i}(\d y^{[i]}) \otimes \mu^{(i,n]}(\d y^{(i,n]}|x^{[i-1]}x_i) \\
	\tilde{\nu}(\d y) & \deq \delta_{x^{[i-1]}z_i}(\d y^{[i]}) \otimes \mu^{(i,n]}(\d y^{(i,n]}|x^{[i-1]}z_i).
\end{align*}
Using this definition and Lemma~\ref{lm:Dobrushin}, we can write
\begin{align}
	&\int_{\sX^{(i,n]}} f(x^{[i-1]}z_iy^{(i,n]})\mu^{(i,n]}(\d y^{(i,n]}|x^{[i-1]}x_i) - \int_{\sX^{(i,n]}} f(x^{[i-1]}z_iy^{(i,n]})\mu^{(i,n]}(\d y^{(i,n]}|x^{[i-1]}z_i) \nonumber\\
	& \qquad = \int f \d\nu - \int f\d\tilde{\nu} \nonumber\\
	& \qquad \le \sum_{j,k \in T} \delta_j(f) D^\nu_{jk} b^{\nu,\tilde{\nu}}_k.\label{eq:deltaK_2}
\end{align}
It remains to obtain explicit upper bounds on the entries of $D^\nu$ and $b^{\nu,\tilde{\nu}}$. To that end, we first note that, for a given $j \in T$ and for any $u,y \in \sX$,
\begin{align*}
	W_{j}\big( \nu_j(\cdot|u^{T \bs \{j\}}), \nu_j(\cdot|y^{T \bs \{j\}}) \big) = \begin{cases}
	0, & j \le i \\
	W_{j}\big( \mu_j(\cdot|x^{[i-1]}x_iu^{(i,n] \bs \{j\}}), \mu_j(\cdot|x^{[i-1]}z_iu^{(i,n] \bs \{j\}})\big), & j > i
	\end{cases}.
\end{align*}
Therefore, $C^\nu_{jk} \le C^\mu_{jk}$. Likewise, for a given $k \in T$ and for any $y \in \sX$,
\begin{align*}
	W_{k}\big( \nu_k(\cdot|y^{T\bs\{k\}}), \tilde{\nu}_k(\cdot | y^{T\bs\{k\}}) \big) = \begin{cases}
	0, & k \le i \\
	W_{k}\big( \mu_k(\cdot|x^{[i-1]}x_iy^{(i,n] \bs \{k\}}), \mu_k(\cdot|x^{[i-1]}z_i y^{(i,n] \bs \{k\}})\big), & k > i
	\end{cases}
\end{align*}
Therefore, $b^{\nu,\tilde{\nu}}_k  \le C^\mu_{ki}\rho_i(x_i,z_i)$. Since the matrices $C^\nu$ and $C^\mu$ are nonnegative, $D^\nu_{jk} \le D^\mu_{jk}$. Consequently, we can write
\begin{align}
	\int f \d\nu - \int f\d\tilde{\nu} &\le \sum_{j,k \in T} \delta_j(f) D^\mu_{jk} C^\mu_{ki} \rho_i(x_i,z_i) \nonumber\\
	&= \sum_{j \in T} \delta_j(f) (D^\mu C^\mu)_{ji} \rho_i(x_i,z_i) \nonumber\\
	&= \sum_{j \in T} \delta_j(f) (D^\mu - {\rm id})_{ji} \rho_i(x_i,z_i). \label{eq:deltaK_3}
\end{align}
Therefore, from \eqref{eq:deltaK_2} and \eqref{eq:deltaK_3}, we have
\begin{align}
	\frac{K^{(i+1)} f(x) - K^{(i+1)} f(z)}{\rho_i(x_i,z_i)} &\le \delta_i(f) + \sum_{j \in T} (D^\mu - \1)^\trn_{ij} \delta_j(f) \nonumber\\
	&= \sum_{j \in T} (D^\mu)^\trn_{ij} \delta_j(f).
\end{align}
We have thus proved the following:
\begin{corollary} Suppose that the probability measure $\mu$ satisfies the Dobrushin uniqueness condition, i.e.,  the spectral radius of its Dobrushin interdependence matrix $C^\mu$ is strictly smaller than unity. Then, for any $t \ge 0$, the concentration inequality \eqref{eq:martingale_bound} holds with
	\begin{align}
		\Gamma_{ij} = \| \rho_i \| (D^\mu)^\trn_{ij}, \qquad i,j \in T.
	\end{align}
\end{corollary}
For example, when each $\sX_i$ is equipped with the trivial metric $\rho_i(x_i,z_i) = \1\{x_i \neq z_i\}$, we have $\| \rho_i \| = 1$ for all $i$, and consequently obtain the concentration inequality
\begin{align}
  \label{dobrushin-kulske}
	\PP_\mu\Big\{ |f(X) - \E_\mu[f(X)]| \ge t \Big\} \le 2 \exp\left(-\frac{2t^2}{\| (D^\mu)^\trn \bdelta(f) \|^2_{\ell^2(T)}}\right).
\end{align} 
The same inequality, but with a worse constant in the exponent, was obtained by K\"ulske \cite[p.~45]{kulske03}.

\subsection{Concentration inequalities via couplings}

Another method for constructing Wasserstein matrices for the Markov kernels $K^{(i)}$ is via couplings. One notable advantage of this method is that it does not explicitly rely on the Dobrushin uniqueness condition; however, some such condition is typically necessary in order to obtain good bounds for the norm $\| \Gamma \bdelta(f) \|_{\ell^2(T)}$.

Fix an index $i \in T$ and any two $x,z \in \sX$ that differ only in the $i$th coordinate: $x^{T \bs \{i\}} = z^{T \bs \{i\}}$ and $x_i \neq z_i$. Let ${{\bf P}}^{[i]}_{x,z}$ be any coupling of the conditional laws $\mu^{(i,n]}(\cdot|x^{[i]})$ and $\mu^{(i,n]}(\cdot|z^{[i]})$. Then for any $f \in L^1(\mu)$ we can write
\begin{align*}
& K^{(i+1)} f(x) - K^{(i+1)} f(z) \\
&\qquad= \int_{\sX^{(i,n]} \times \sX^{(i,n]}} {{\bf P}}^{[i]}_{x,z}(\d u^{(i,n]}, \d y^{(i,n]}) \left(f(x^{[i]},u^{(i,n]})-f(z^{[i]},y^{(i,n]})\right) \\
&\qquad\le \delta_i(f) \rho_i(x_i,z_i) + \sum_{j \in T:\, j > i} \delta_j(f) \int_{\sX^{(i,n]} \times \sX^{(i,n]}} {{\bf P}}^{[i]}_{x,z}(\d u^{(i,n]},\d y^{(i,n]}) \rho_j(u_j,y_j).
\end{align*}
Therefore,
\begin{align*}
	\frac{|K^{(i+1)}f(x)-K^{(i+1)}f(z)|}{\rho_i(x_i,z_i)} &\le \delta_i(f) + \sum_{j \in T:\, j > i} \frac{\int \rho_j \d{{\bf P}}^{[i]}_{x,z}}{\rho_i(x_i,z_i)} \delta_j(f) \\
	&\le \delta_i(f) + \sum_{j \in T:\, j > i} \sup_{x,z \in \sX \atop x^{T \bs \{i\}} = z^{T \bs \{i\}}}\frac{\int \rho_j \d{{\bf P}}^{[i]}_{x,z}}{\rho_i(x_i,z_i)} \delta_j(f).
\end{align*}
Remembering that we only need the $i$th row of a Wasserstein matrix for $K^{(i+1)}$, we may take
\begin{align}\label{eq:coupling_V}
	V^{(i+1)}_{ij} = \begin{cases}
	0, & i > j \\
	1, & i = j \\
	 \displaystyle\sup_{x,z \in \sX \atop x^{T \bs \{i\}} = z^{T \bs \{i\}}}\frac{\displaystyle\int \rho_j \d{{\bf P}}^{[i]}_{x,z}}{\rho_i(x_i,z_i)}, & i < j
	 \end{cases}.
\end{align}
We have thus proved the following:
\begin{corollary} For each index $i \in T$ and for each pair $x,z \in \sX$ of configurations with $x^{T \bs \{i\}} = z^{T \bs \{i\}}$, pick an arbitrary coupling ${{\bf P}}^{[i]}_{x,z}$ of the conditional laws $\mu^{\{i\}}(\cdot|x^{[i]})$ and $\mu^{\{i\}}(\cdot|z^{[i]})$. Then, for any $t \ge 0$, the concentration inequality \eqref{eq:martingale_bound} holds with 
	\begin{align}\label{eq:Gamma}
		\Gamma_{ij} = \| \rho_i \| V^{(i+1)}_{ij}, \qquad i,j \in T
	\end{align}
	where the entries $V^{(i+1)}_{ij}$ are given by \eqref{eq:coupling_V}.
\end{corollary}
In the case when each $\sX_i$ is equipped with the trivial metric $\rho_i(x_i,z_i) = \1\{x_i \neq z_i\}$, the entries $\Gamma_{ij}$ for $j > i$ take the form
\begin{align}
	\Gamma_{ij} = \sup_{x,z \in \sX \atop x^{T \bs \{i\}} = z^{T \bs \{i\}}} {{\bf P}}^{[i]}_{x,z}\left\{Y^{(0)}_j \neq Y^{(1)}_j\right\},
\end{align}
where $(Y^{(0)},Y^{(1)}) = \big((Y^{(0)}_{i+1},\ldots,Y^{(0)}_n),(Y^{(1)}_{i+1},\ldots,Y^{(1)}_n)\big)$ is a random object taking values in $\sX^{(i,n]} \times \sX^{(i,n]}$. A special case of this construction, under the name of \textit{coupling matrix}, was used by
    Chazottes et al.~\cite{chazottes07}.
    In that work, each ${{\bf P}}^{[i]}_{x,z}$ was chosen to minimize
    \begin{align*}
{{\bf P}} \{Y^{(0)} \neq Y^{(1)} \},
\end{align*}
over all couplings ${{\bf P}}$ of $\mu^{(i,n]}(\cdot|x^{[i]})$ and $\mu^{(i,n]}(\cdot|z^{[i]})$, in which case we have
\begin{align*}
	{{\bf P}}^{[i]}_{x,z}\{Y^{(0)} \neq Y^{(1)}\} &= \inf_{{\bf P} \in \cC(\mu^{(i,n]}(\cdot|x^{[i]}),\mu^{(i,n]}(\cdot|z^{[i]}))} {{\bf P}}\{Y^{(0)} \neq Y^{(1)}\} \\
&= \tvnrm{\mu^{(i,n]}(\cdot|x^{[i]}) - \mu^{(i,n]}(\cdot|z^{[i]})}.
\end{align*}
\sloppypar\noindent However, it is not clear how to relate the quantities ${{\bf P}}^{[i]}_{x,z}\left\{Y^{(0)}_j \neq Y^{(1)}_j\right\}$ and ${{\bf P}}^{[i]}_{x,z}\left\{Y^{(0)} \neq Y^{(1)}\right\}$, apart from the obvious bound
\begin{align*}
	{{\bf P}}^{[i]}_{x,z} \{Y^{(0)}_j \neq Y^{(1)}_j \} &\le {{\bf P}}^{[i]}_{x,z} \{Y^{(0)} \neq Y^{(1)} \} = \tvnrm{\mu^{(i,n]}(\cdot|x^{[i]}) - \mu^{(i,n]}(\cdot|z^{[i]})},
\end{align*}
which gives
\begin{align*}
	\Gamma_{ij} \le \sup_{x,z \in \sX \atop x^{T \bs \{i\}} = z^{T \bs \{i\}}} \tvnrm{\mu^{(i,n]}(\cdot|x^{[i]}) - \mu^{(i,n]}(\cdot|z^{[i]})}.
\end{align*}
An alternative choice of coupling is the so-called \textit{maximal coupling} due to Goldstein \cite{Goldstein_maximal}, which for our purposes can be described as follows: let $U = (U_\ell)^m_{\ell=1}$ and $Y = (Y_\ell)^m_{\ell =1}$ be two random $m$-tuples taking values in a product space $\sE = \sE_1 \times \ldots \times \sE_m$, where each $\sE_\ell$ is Polish. Then there exists a coupling ${{\bf P}}$ of the probability laws  $\cL(U)$ and $\cL(Y)$, such that
\begin{align}\label{eq:Goldstein}
	{{\bf P}} \left\{ U^{[\ell,m]} \neq Y^{[\ell,m]} \right\} = \tvnrm{\cL(U^{[\ell,m]})-\cL(Y^{[\ell,m]})}, \qquad \ell \in \{1,\ldots,m\}.
\end{align}
Thus, for each $i \in T$ and for every pair $x,z \in \sX$ with $x^{T \bs \{i\}} = z^{T \bs \{i\}}$, let ${{\bf P}}^{[i]}_{x,z}$ be the Goldstein coupling of $\mu^{(i,n]}(\cdot|x^{[i]})$ and $\mu^{(i,n]}(\cdot|z^{[i]})$. Then for each $j \in \{i+1,\ldots,n\}$, using \eqref{eq:Goldstein} we have
\begin{align*}
	{{\bf P}}^{[i]}_{x,z}\{Y^{(0)}_j \neq Y^{(1)}_j \} &\le {{\bf P}}^{[i]}_{x,z}\{ (Y^{(0)}_j,\ldots,Y^{(0)}_n) \neq (Y^{(1)}_j,\ldots,Y^{(1)}_n) \} \\
	&= \tvnrm{\mu^{[j,n]}(\cdot|x^{[i]}) - \mu^{[j,n]}(\cdot|z^{[i]})}.
\end{align*}
This choice of coupling gives rise to the upper-triangular matrix $\Gamma = (\Gamma_{ij})_{i,j \in T}$ with
\begin{align}
  \label{GammaKontRam}
	\Gamma_{ij} = \begin{cases}
	0, & i > j \\
	1, & i = j \\
	\displaystyle\sup_{x,z \in \sX \atop x^{T \backslash \set{i}} = z^{T \backslash \set{i}}}
        \tvnrm{\mu^{[j,n]}(\cdot|x^{[i]})-\mu^{[j,n]}(\cdot|z^{[i]})}, & i < j
	\end{cases}.
\end{align}	
Substituting this matrix into \eqref{eq:martingale_bound}, we recover the concentration inequality of
Kontorovich and Ramanan \cite{kontram06},
but with an improved constant in the exponent.

\begin{remark}{
  It was erroneously claimed in
  \cite{kont07-thesis,kontorovich12,kont-weiss-14}
  that the basic concentration inequalities of Chazottes et al.~\cite{chazottes07}
  and Kontorovich and Ramanan \cite{kontram06} are essentially the same,
  only derived using different methods.
  As the discussion above elucidates, the two methods use different couplings (the
  former, explicitly, and the latter, implicitly) --- which yield quantitatively
  different and, in general, incomparable mixing coefficients.}
 \end{remark}
 
 \begin{remark}
   {Kontorovich and Ramanan
     obtained the matrix \eqref{GammaKontRam} using analytic methods without constructing an explicit coupling. In 2012, S.~Shlosman posed the following question: could this matrix have been derived using a suitable coupling? We can now answer his question in the affirmative: the coupling  is precisely Goldstein's maximal coupling.}
\end{remark}  

As an illustration, let us consider two specific types of the probability law $\mu$: a directed Markov model (i.e., a Markov chain) and an undirected Markov model (i.e., a Gibbsian Markov random field). In the directed case, suppose that the elements of $T$ are ordered in such a way that $\mu$ can be disintegrated in the form
\begin{align}
  \label{MarkovKK}
\mu(\d x) = \mu_1(\d x_1) \otimes K_1(x_1,\d x_2) \otimes K_2(x_2,\d x_3) \otimes \ldots \otimes K_{n-1}(x_{n-1}, \d x_n),
\end{align}
where $\mu_0$ is a Borel probability measure on $(\sX_1,\cB_1)$, and, for each $i \in [1,n-1]$, $K_i$ is a Markov kernel from $\sX_i$ to $\sX_{i+1}$. For each $i \in [1,n)$, let
$$
\theta_i \deq \sup_{x_i,z_i \in \sX_i} \tvnrm{K_i(x_i,\cdot) - K_i(z_i,\cdot)}
$$
be the Dobrushin contraction coefficient of $K_i$. Fix $1\le i<j\le n$ and $x^{[i-1]}\in\sX^{[i-1]},y_i,y_i'\in\sX_i$.
An easy calculation \cite{kontorovich12} shows that,
defining the signed measures $\eta_i$ on $\sX_{i+1}$ by $\eta(\d x_{i+1})=K_i(y_i,\d x_{i+1})-K_i(y_i',\d x_{i+1})$
and $\zeta_j$ on $\sX_j$ by
$$ \zeta_j = \eta_i K_i K_{i+1} K_{i+2}\ldots K_{j-1},$$
we have
\begin{align}
  \label{contr-decomp}
  \tvnrm{\mu^{[j,n]}(\cdot|x^{[i-1]}y_i)-\mu^{[j,n]}(\cdot|x^{[i-1]}y_i')}
  &=  \tvnrm{\zeta_j} \le  
  \theta_i\theta_{i+1}\ldots\theta_{j-1},
\end{align}
where (\ref{eq:doblin}) was repeatedly invoked to obtain the last inequality.
The above yields an upper bound on the $\Gamma_{ij}$ in (\ref{GammaKontRam})
and hence in the corresponding concentration inequality in
\eqref{eq:martingale_bound}.
When more delicate (e.g., spectral \cite{kont-weiss-14,Paulin15}) estimates on $\tvnrm{\zeta_j}$ are available,
these translate directly into tighter concentration bounds.

In the undirected case, $\mu$ is a Gibbsian Markov random field induced by pair potentials \cite{Georgii}. To keep things simple, we assume that the local spaces $\sX_i$ are all finite. Define an undirected graph with vertex set $T=[n]$ and edge set $E=\set{[i,i+1]:1\le i<n}$ (i.e., a chain graph with vertex set $T$). Associate with each edge $(i,j)\in E$ a potential function $\psi_{ij}:\sX_i\times \sX_j\to[0,\infty)$.
  Together, these define a probability measure $\mu$ on $\sX$  via
  $$ \mu(x) = \frac{
    \prod_{(i,j)\in E}\psi_{ij}(x_i,x_j)
  }{
    \sum_{y\in\sX}
    \prod_{(i,j)\in E}\psi_{ij}(y_i,y_j)
  }.
  $$
  Since $\mu$ is a Markov measure on $\sX$, there is a sequence of Markov kernels
  $K_1,\ldots, K_{n-1}$ generating $\mu$ in the sense of (\ref{MarkovKK}). It is shown in
  \cite{kontorovich12} that the contraction coefficient $\theta_i$ of the kernel $K_i$ is bounded by
  \begin{align*}
    \theta_i\le\frac{R_i-r_i}{R_i+r_i},
  \end{align*}
  where
  $$
  R_i=\sup_{(x_i,x_{i+1})\in\sX_i\times\sX_{i+1}} \psi_{i,i+1}(x_i,x_{i+1}),
  \qquad
  r_i=\inf_{(x_i,x_{i+1})\in\sX_i\times\sX_{i+1}} \psi_{i,i+1}(x_i,x_{i+1}).$$
  The estimate above implies a concentration result, either via (\ref{contr-decomp})
  or via (\ref{dobrushin-kulske}). To apply the latter, recall that
  $D^\mu =\sum_{k=1}^\infty (C^\mu)^k$, where $C^\mu$ is the Dobrushin interdependence matrix
  defined in (\ref{dobrushin-C-def}).
  Assuming that
  $\rho$ is the unweighted Hamming metric (i.e., $\rho_i(x_i,z_i) = \1\{x_i \neq z_i\}$ for all $i$) and that
  the $\theta_i$'s are all majorized by some $\theta<1$,
  it is easy to see that  $(C^\mu)_{ij}\le \theta^{|i-j|}$.

\section{Open questions}
\label{sec:open}

Our focus in this chapter has been on the martingale method for establishing concentration inequalities. In the case of product measures, other techniques, such as the entropy method or transportation-information inequalities, often lead to sharper bounds. However, these alternative techniques are less developed in the dependent setting, and there appears to be a gap between what is achievable using the martingale method and what is achievable using other means. We close the chapter by listing some open questions that are aimed at closing this gap:
\begin{itemize}
	\item \textbf{(Approximate) tensorization of entropy.} In the independent case, it is possible to derive the same concentration inequality (e.g., McDiarmid's inequality) using either the martingale method or the entropy method, often with the same sharp constants. However, once the independence assumption is dropped, the situation is no longer so simple. Consider, for example, tensorization of entropy. Several authors (see, e.g., \cite{Marton_tensorization_0,Caputo_etal_tensorization,Marton_tensorization}) have obtained so-called \textit{approximate} tensorization inequalities for the relative entropy in the case of weakly dependent random variables: under certain regularity conditions on $\mu$, there exists a constant $A_\mu \ge 1$, such that, for any other probability measure $\nu$,
\begin{align}\label{eq:approximate_tensorization}
	D(\nu \| \mu) \le A_\mu \cdot \sum_{i \in T} \E_\nu D\big(\nu_i(\cdot|X^{T\backslash\set{i}}) \big\| \mu_i(\cdot|X^{T\backslash\set{i}})\big).
\end{align}
Having such an inequality in hand, one can proceed to prove concentration for Lipschitz functions in exactly the same way as in the independent case. However, it seems that the constants $A_\mu$ in \eqref{eq:approximate_tensorization} are not sharp in the sense that the resulting concentration inequalities are typically worse than what one can obtain using Theorems~\ref{thm:mart_1} or \ref{thm:mart_2} under the same assumptions on $\mu$ and $f$. This motivates the following avenue for further investigation: Derive sharp inequalities of the form \eqref{eq:approximate_tensorization} by relating the constant $A_\mu$ to appropriately chosen Wasserstein matrices.
	\item \textbf{General Wasserstein-type matrices.} Using the techniques pioneered by Marton, Samson proved the following concentration of measure result: Consider a function $f : \sX \to \Reals$ satisfying an ``asymmetric'' Lipschitz condition of the form
	$$
	f(x) - f(y) \le \sum_{i \in T} \alpha_i(x) \1\{x_i \neq y_i\}, \qquad \forall x,y \in \sX
	$$
for some functions $\alpha_i : \sX \to \Reals$, such that $\sum_{i \in T}\alpha^2_i(x) \le 1$ for all $x \in \sX$. Then, for any Borel probability measure $\mu$ on $\sX$, we have
\begin{align}\label{eq:Samson}
	\PP_\mu\Big\{ f(X)-\E_\mu[f(X)] \ge t\Big\} \le \exp\left(-\frac{t^2}{2 \| \Delta \|^2_2}\right),
\end{align}
where the matrix $\Delta$ has entries $\Delta_{ij} = \sqrt{\Gamma_{ij}}$ with $\Gamma_{ij}$ given by \eqref{GammaKontRam}, and
$$
\| \Delta \|_2 \deq \sup_{v \in \Reals^T \bs \{0\}} \frac{\| \Delta v \|_{\ell^2(T)}}{\| v \|_{\ell^2(T)}}
$$
is the operator norm of $\Delta$. A more general result in this vein was derived by Marton \cite{marton03}, who showed that an inequality of the form \eqref{eq:Samson} holds with $\Delta$ computed in terms of any matrix $\Gamma$ of the form \eqref{eq:coupling_V}, where each $\rho_i$ is the trivial metric. Samson's proof relies on a fairly intricate recursive coupling argument. It would be interesting to develop analogs of \eqref{eq:Samson} for arbitrary choices of the metrics $\rho_i$ and with full freedom to choose the Wasserstein matrices $V^{(i)}$ for each $i \in T$. A recent paper by Wintenberger \cite{Wintenberger} pursues this line of work.
\item \textbf{The method of exchangeable pairs and Wasserstein matrices.}
  An alternative route towards concentration inequalities in the dependent setting is via Stein's method of exchangeable pairs \cite{chatterjee05,Chatterjee_PTRF}. Using this method, Chatterjee obtained the following result \cite[Chap.~4]{chatterjee05}: Let $f : \sX \to \Reals$ be a function which is $1$-Lipschitz with respect to the weighted Hamming metric $\rho_{\boldsymbol{\alpha}}$ defined in \eqref{weighted-ham-def}. Let $\mu$ be a Borel probability measure on $\sX$, whose Dobrushin interdependence matrix $C^\mu$ satisfies the condition $\| C^\mu \|_2 < 1$. Then, for any $t \ge 0$,
	\begin{align}\label{eq:chatterjee}
		\PP_\mu\Big\{ |f(X)-\E_\mu[f(X)] \ge t|\Big\} \le 2\exp\left(-\frac{(1-\| C^\mu \|_2)t^2}{\sum_{i \in T}\alpha^2_i}\right).
	\end{align}
The key ingredient in the proof of \eqref{eq:chatterjee} is the so-called Gibbs sampler, i.e., the Markov kernel $\bar{K}$ on $\sX$ given by
$$
\bar{K}(x, \d y) \deq \frac{1}{|T|}\sum_{i \in T} \delta_{x^{T\bs\set{i}}}(\d y^{T\bs\set{i}}) \otimes \mu_i(\d y_i|x^{T\bs\set{i}}).
$$
This kernel leaves $\mu$ invariant, i.e., $\mu = \mu \bar{K}$, and it is easy to show (see, e.g., \cite{Marton_tensorization}) that it contracts the $\bar{W}$ distance: for any other probability measure $\nu$ on $\sX$,
\begin{align*}
	\bar{W}(\mu \bar{K},\nu \bar{K}) \le \left(1-\frac{1-\| C^\mu \|_2}{|T|}\right) \bar{W}(\mu,\nu).
\end{align*}
Since one can obtain contraction estimates for Markov kernels using Wasserstein matrices, it is natural to ask whether Chatterjee's result can be derived as a special case of a more general method, which would let us freely choose an arbitrary Markov kernel $K$ that leaves $\mu$ invariant and control the constants in the resulting concentration inequality by means of a judicious choice of a Wasserstein matrix for $K$. Such a method would most likely rely on general comparison theorems for Gibbs measures \cite{rebeschini14}.

\end{itemize}

\section*{Acknowledgments}

The second author would like to thank IMA for an invitation to speak at the workshop on Information Theory and Concentration Phenomena in Spring 2015, which was part of the annual program ``Discrete Structures: Analysis and Applications.'' The authors are grateful to the anonymous referee for several constructive suggestions, and to Dr.~Naci Saldi for spotting an error in an earlier version of the manuscript. A.~Kontorovich was partially supported by
the Israel Science Foundation
(grant No.~1141/12)
and 
a Yahoo Faculty award. M.~Raginsky would like to acknowledge the support of the U.S.\ National Science Foundation via CAREER award CCF--1254041.

\end{document}